\documentclass[a4paper]{article}
\usepackage[T1]{fontenc}
\usepackage[utf8]{inputenc}
\usepackage[babel]{csquotes}
\usepackage{microtype}
\usepackage{geometry} 
\usepackage{quoting} 
\quotingsetup{font=small}
\usepackage{amsmath}
\usepackage{amsfonts}
\usepackage{amssymb}
\usepackage{amsthm}
\makeatletter
\renewcommand{\pod}[1]{\allowbreak\mathchoice
	{\if@display \mkern 18mu\else \mkern 8mu\fi (#1)}
	{\if@display \mkern 18mu\else \mkern 8mu\fi (#1)}
	{\mkern4mu(#1)}
	{\mkern4mu(#1)}
}
\newcommand{\lr}[1]{\lfloor#1\rfloor}
\usepackage{mathrsfs}
\usepackage{mathtools}
\usepackage{braket}
\usepackage{bm}
\usepackage{enumerate}
\usepackage{empheq}
\usepackage{emptypage}
\usepackage{graphics}
\usepackage{tikz}
\usetikzlibrary{matrix,arrows,decorations.pathmorphing}
\usepackage{tikz-cd}
\usepackage{xparse}
\usepackage{stmaryrd}
\usepackage{latexsym,xspace,centernot}
\usepackage{pgf}
\usepackage{authblk}

\newcommand{\C}{\mathbb{C}}
\newcommand{\Q}{\mathbb{Q}}

\newcommand{\m}[4]{\begin{pmatrix}
		#1&#2\\#3&#4
	\end{pmatrix}}
	
	\renewcommand{\epsilon}{\varepsilon}
	\newcommand{\g}[1]{\mathfrak{#1}}
	\theoremstyle{plain}
	
	\newtheorem{theorem}{Theorem}
	\newtheorem{lemma}{Lemma}

	\theoremstyle{remark}
	\newtheorem{remark}{Remark}
	\theoremstyle{conjecture}
	
	\theoremstyle{definition}

	\renewcommand{\pod}[1]{\allowbreak\mathchoice
		{\if@display \mkern 18mu\else \mkern 8mu\fi (#1)}
		{\if@display \mkern 18mu\else \mkern 8mu\fi (#1)}
		{\mkern4mu(#1)}
		{\mkern4mu(#1)}
	}
	\DeclareMathOperator{\norm}{Norm}
	
\begin{document}
		\title{\textbf{A note on the zeros of generalized Hurwitz zeta functions}}\date{}
		\author{Giamila Zaghloul}
		\affil{Dipartimento di Matematica\\Università degli Studi di Genova\\via Dodecaneso 35, 16146 Genova}
		\maketitle
		\begin{abstract}
			Given a function $f(n)$ periodic of period $q\geq 1$ and an irrational number $0<\alpha\leq 1$, Chatterjee and Gun (cf. \cite{c-g}) proved that the series $F(s,f,\alpha)=\sum_{n=0}^{\infty}\frac{f(n)}{(n+\alpha)^s}$ has infinitely many zeros for $\sigma>1$ when $\alpha$ is transcendental and $F(s,f,\alpha)$ has a pole at $s=1$, or when $\alpha$ is algebraic irrational and $c=\frac{\max{f(n)}}{\min{f(n)}}<1.15$. In this note, we prove that the result holds in full generality. 
		\end{abstract}
		\section{Introduction} 
		Let $0<\alpha\leq 1$ be a real number, the \emph{Hurwitz zeta function} is defined as
		\[
		\zeta(s,\alpha)=\sum_{n=0}^{\infty}\frac{1}{(n+\alpha)^{s}},
		\]
		for $s=\sigma+it\in\C$ with $\sigma>1$. It is known that it admits a meromorphic continuation to $\C$ with a simple pole at $s=1$. In their paper \cite{d-h},
		Davenport and Heilbronn proved that if $\alpha\notin\{ 1,\frac{1}{2}\}$ is either rational or transcendental, then $\zeta(s,\alpha)$ has infinitely many zeros for $\sigma>1$. The same result when $\alpha$ is algebraic irrational was proved by Cassels in \cite{cas}.
		
		Let now $f(n)$ be a periodic function of period $q\geq 1$. For $\sigma>1$, we define the generalized Hurwitz zeta function as 
		\[
		F(s,f,\alpha)=\sum_{n=0}^{\infty}\frac{f(n)}{(n+\alpha)^{s}}.
		\]
		As for $\zeta(s,\alpha)$, $F(s,f,\alpha)$ is absolutely convergent for $\sigma>1$ and it admits a meromorphic continuation to the whole complex plane (see e.g. \cite{c-g}).
		
		In \cite{c-g}, Chatterjee and Gun assume that $f(n)$ is positive valued and prove that $F(s,f,\alpha)$ has infinitely many zeros in the half-plane $\sigma>1$ if $\alpha$ is transcendental and $F(s,f,\alpha)$ has a pole at $s=1$, or if $\alpha$ is algebraic irrational and 
		\begin{equation}\label{max}
		c:=\frac{\underset{n}{\max}f(n)}{\underset{n}{\min}f(n)}<1.15.
		\end{equation}
		In this note we show that these assumptions can be removed, proving the result in full generality, also including the case of $\alpha$ rational, which can be easily deduced from \cite{s-w}.

		\begin{theorem}\label{theor}
			Let $f(n)$ be a non-identically zero periodic function with period $q\geq 1$ and let $0<\alpha\leq 1$ be a real number. If $\alpha\notin\{1,\frac{1}{2}\}$, or if $\alpha\in\{1,\frac{1}{2}\}$ and $F(s,f,\alpha)$ is not of the form $P(s)L(s,\chi)$, where $P(s)$ is a Dirichlet polynomial and $L(s,\chi)$ is the $L$-function associated to a Dirichlet character $\chi$, then $F(s,f,\alpha)$ has infinitely many zeros with $\sigma>1$. 
		\end{theorem}
		Observe that if $\alpha=1$, $F(s,f,1)$ reduces to a Dirichlet series with periodic coefficients. By the result of Saias and Weingartner \cite[Corollary]{s-w}, we know that it does not vanish in the half-plane $\sigma>1$ if and only if it is the product of a Dirichlet polynomial and a Dirichlet $L$-function. \begin{remark}
			Examples of functions $f(n)$ giving rise to non-vanishing series in the right half-plane are $f(n)=\chi(n+1)$, where $\chi$ is a Dirichlet character$\mod q$, or $f(n)=(-1)^n$.
		\end{remark}
		If $0<\alpha<1$ is rational, $F(s,f,\alpha)$ can be written as a linear combination of Dirichlet $L$-function, 
		\begin{equation}\label{dec}
		F(s,f,\alpha)=\sum_{\chi\in\mathcal{C}}P_{\chi}(s)L(s,\chi),
		\end{equation} 
		where $\mathcal{C}$ is a set of primitive characters and $P_{\chi}(s)$ is a Dirichlet polynomial. Again by \cite{s-w}, expression \eqref{dec} does not vanish in the half-plane $\sigma>1$ if and only if the sum reduces to a single term. Let now $\alpha=\frac{a}{b}\in\Q$, with $(a,b)=1$, $1\leq a< b$. Then, 
		\begin{equation}\label{rat}		F(s,f,a/b)=b^s\sum_{n=0}^{\infty}\frac{f(n)}{(bn+a)^s}=b^s\sum_{m\equiv a\pmod b}\frac{g(m)}{m^s},
		\end{equation}
		where $g(m)$ is periodic of period $bq$. We prove the following lemma. 
		\begin{lemma}\label{lemma_rat}
			Let $\alpha=\frac{a}{b}$, with $(a,b)=1$, $1\leq a<b$. If $\frac{a}{b}\ne \frac{1}{2}$, then $F(s,f,\frac{a}{b})$ is not of the form $P(s)L(s,\chi)$, where $P$ is a Dirichlet polynomial and $L(s,\chi)$ is the Dirichlet $L$-function associated to the character $\chi$.  
		\end{lemma}
		\begin{proof}
			Consider a Dirichlet polynomial $P(s)=\sum_{n\in\mathcal{N}}\frac{a(n)}{n^s}$,
			where $\mathcal{N}$ is a non-empty finite set of positive integers, and let $\chi$ be a Dirichlet character$\mod k$. Then, 
			\[
			P(s)L(s,\chi)=\sum_{m}\frac{b(m)}{m^s},\quad\text{where}\quad b(m)=\sum_{\substack{n\in\mathcal{N}\\n\mid m}}a(n)\chi\bigg(\frac{m}{n}\bigg),
			\]
			and the coefficients $b(m)$ are periodic of period $k\prod_{n\in\mathcal{N}}n$. Assume that there exist two coprime integers $h<r$, such that $b(m)\ne 0$ only if $m\equiv h\pmod r$. Let $n_1:=\min\mathcal{N}$, then $b(n_1)=a(n_1)\ne 0$ and so $n_1\equiv h\pmod r$. On the other hand, $b(-n_1)=\chi(-1)a(n_1)\ne 0$, then $-n_1\equiv h\pmod r$. It follows that $2h\equiv 0\pmod r$, which implies $r=2$. Thus, we conclude that expression \eqref{rat} can be of the form $P(s)L(s,\chi)$ only if $\alpha=\frac{1}{2}$. 
		\end{proof}
		Observe that if $\alpha=\frac{1}{2}$, the sum \eqref{dec} reduces to a single term for instance if $g(m)=c\chi(m)$, where $\chi$ is a Dirichlet character$\mod 2q$ and $c$ is a non-zero constant (i.e. $f(n)=c\chi(2n+1)$). In this case, $F(s,f,\frac{1}{2})=c2^sL(s,\chi)\ne 0$ in $\sigma>1$.\\
		
		If $\alpha$ is transcendental, the argument of Davenport and Heilbronn (cf. \cite{d-h}) for the Hurwitz zeta function applies also to $F(s,f,\alpha)$. Indeed, we have
		\begin{equation}\label{limit}
		\sum_{n=0}^{\infty}\frac{|f(n)|}{(n+\alpha)^{\sigma}}\to +\infty\quad\text{as}\quad \sigma\to 1^+.
		\end{equation}
		Then, the assumption on the existence of the pole can be avoided and one can proceed as in \cite{d-h} or \cite{c-g}.    
		Thus, we focus on the case of $\alpha$ algebraic irrational. The proof of the theorem in this case is based on a modification of Cassels' original lemma (see \cite{cas}). A suitable decomposition over the residue classes allows us to remove the assumption \eqref{max}.\\
		\section{Proof of the theorem}
		As observed, we can assume that $\alpha$ is algebraic irrational. Let $K=\Q(\alpha)$ and let $O_{K}$ be its ring of integers. Denote by $\g{a}$ the denominator ideal of $\alpha$, i.e. $\g{a}=\Set{r\in O_{K}\mid r\cdot(\alpha)\subseteq O_{K}}$, where $(\alpha)$ is the principal fractional ideal generated by $\alpha$. Then for any integer $n\geq 0$, $(n+\alpha)\g{a}$ is an integral ideal. The following result holds.  
		\begin{lemma}\label{lemma}
			Let $0<\alpha<1$ be an algebraic irrational number and let $K=\Q(\alpha)$. Given a positive integer $q$, fix $b\in\Set{0,\dots,q-1}$.
			There exists an integer $N_0>10^6q$, depending on $\alpha$ and $q$, satisfying the following property:\\
			for any integer $N>N_0$ put $M=\lfloor 10^{-6}N \rfloor$, then at least $0.54\frac{M}{q}$ of the integers $n\equiv b\pmod q$, $N<n\leq N+M$ are such that $(n+\alpha)\g{a}$ is divisible by a prime ideal $\g{p}_n$ for which 
			\[
			\g{p}_n\nmid \prod_{\substack{m\leq N+M\\m\ne n}}(m+\alpha)\g{a}.
			\]
		\end{lemma}
		In the following sections, we first show how to complete the proof of Theorem \ref{theor} assuming the above lemma and then we give a proof of the lemma itself. 
		\subsection{Proof of the main result}
		We rearrange Cassels' argument with some suitable small modifications. As in \cite{cas}, or directly by Bohr's theory (see \cite[Theorem 8.16]{a}), it suffices to show that for any $0<\delta<1$ there exist a $\sigma$, with $1<\sigma<1+\delta$, and a completely multiplicative function $\varphi(n):=\varphi((n+\alpha)\g{a})$ of absolute value 1, such that 
		\[
		\sum_{n=0}^{\infty}\frac{f(n)\varphi(n)}{(n+\alpha)^\sigma}=0.
		\]
		Notice that it is enough to define $\varphi(\g{p})$, with $|\varphi(\g{p})|=1$, on the prime ideals $\g{p}$ dividing $(n+\alpha)\g{a}$.
		
		Let $0<\delta<1$, $N_1=\max(N_0,10^{7}q)$ and consider $\sigma$ such that $1<\sigma<1+\delta$ and 
		\begin{equation}\label{N_1}
		\sum_{n=0}^{N_1}\frac{|f(n)|}{(n+\alpha)^{\sigma}}<\frac{1}{100}\sum_{n=N_1+1}^{\infty}\frac{|f(n)|}{(n+\alpha)^{\sigma}}.
		\end{equation}
		Observe that such a $\sigma$ exists by \eqref{limit}. Now, for $\g{p}\mid \g{a}$ or $\g{p}\mid (n+\alpha)\g{a}$ with $n\leq N_1$ we choose $\varphi(\g{p})=1$.\\
		Proceeding by induction, for $j\geq 1$, we put $M_j=\lfloor 10^{-6}N_j\rfloor$ and $N_{j+1}=N_j+M_j$. Suppose we have defined $\varphi(\g{p})$ for any $\g{p}\mid (n+\alpha)\g{a}$ with $n\leq N_j$ in such a way that 
		\begin{equation}\label{induction}
		\bigg| \sum_{n=0}^{N_j}\frac{f(n)\varphi(n)}{(n+\alpha)^{\sigma}}\bigg|<\frac{1}{100}\sum_{n=N_j+1}^{\infty}\frac{|f(n)|}{(n+\alpha)^{\sigma}}.
		\end{equation}
		We want to define $\varphi(\g{p})$ for any prime ideal 
		\begin{equation}\label{primes}
		\g{p}\mid \prod_{n\leq N_{j+1}}(n+\alpha)\g{a}
		\end{equation}
		in such a way that \eqref{induction} holds for $j+1$ in place of $j$. 
		For any $b\in\Set{0,\dots,q-1}$, we divide the integers $N_j<n\leq N_{j+1}$, with $n\equiv b\pmod q$ into two sets $\g{A}(b)$ and $\g{B}(b)$ according to whether a prime ideal $\g{p}_n$ as in Lemma \ref{lemma} exists or not for $N=N_j$ and $M=M_j$. We can easily notice that $|\g{A}(b)|\geq 5$, since 
		\[
		|\g{A}(b)|\geq \frac{54}{100}\frac{M_j}{q}=\frac{54}{100}\frac{\lr{10^{-6}N_j}}{q},
		\]  
		and $N_j\geq 10^7q$. We have then divided the integers $N_j<n\leq N_{j+1}$ into the disjoint sets $
		\g{A}=\cup_{b=0}^{q-1}\g{A}(b)$ and $\g{B}=\cup_{b=0}^{q-1}\g{B}(b)$. As in Cassels', given a prime ideal as in \eqref{primes}, we distinguish three cases:
		\begin{enumerate}
			\item[(1)] $\g{p}\mid \prod_{n\leq N_{j}}(n+\alpha)\g{a}$: in this case $\varphi(\g{p})$ is fixed by the inductive hypothesis.
			\item[(2)] $\g{p}=\g{p}_n$ for some $n\in \g{A}$
			\item[(3)] the remaining $\g{p}$ with property \eqref{primes}. In this case, we fix arbitrarily $\varphi(\g{p})=1$. 			
		\end{enumerate}
		In particular, $\varphi(n)$ is defined for any $n\in\g{B}$, whereas if $n\in\g{A}$, we have that $\varphi(n)=c_n\varphi(\g{p}_n)$, with $c_n$ fixed of modulus 1. Now assume $n\in\g{A}$ and $n\equiv b\pmod q$ with $b\in\Set{0,\dots,q-1}$. Since $f(n)$ is periodic of period $q$ and $|\g{A}(b)|\geq5$, by Bohr's results on addition of convex curves (cf. \cite{b}), for an appropriate choice of $\varphi(\g{p}_n)$ for all $n\in \g{A}(b)$, we have that 
		\[
		\sum_{n\in\g{A}(b)}\frac{f(n)\varphi(n)}{(n+\alpha)^{\sigma}}=\sum_{n\in\g{A}(b)}\frac{f(n)c_n\varphi(\g{p}_n)}{(n+\alpha)^{\sigma}}=f(b)\sum_{n\in\g{A}(b)}\frac{c_n\varphi(\g{p}_n)}{(n+\alpha)^{\sigma}}
		\]
		takes any given value $z$ satisfying
		\[
		|z|\leq S_{3,b}:= |f(b)|\sum_{n\in\g{A}(b)}\frac{1}{(n+\alpha)^{\sigma}}.
		\]
		Let now
		\[
		\Lambda(b):=f(b)\bigg(\sum_{\substack{n\leq N_j\\n\equiv b\pmod q}}\frac{\varphi(n)}{(n+\alpha)^{\sigma}}+\sum_{n\in\g{B}(b)}\frac{\varphi(n)}{(n+\alpha)^{\sigma}}\bigg),
		\]
		and define $\varphi(\g{p}_n)$ for $n\in\g{A}(b)$ so that 
		\[
		\sum_{n\in\g{A}(b)}\frac{f(n)\varphi(n)}{(n+\alpha)^{\sigma}}=\begin{cases}
		-\Lambda(b)\quad&\text{if}\quad |\Lambda(b)|\leq S_{3,b}\\	
		-S_{3,b}\frac{\Lambda(b)}{|\Lambda(b)|}	&\text{if}\quad |\Lambda(b)|>S_{3,b}.	\end{cases}
		\]
		With this choice, it is easy to verify that 
		\begin{equation}\label{N_j+1}
		\bigg|\sum_{\substack{n\leq N_{j+1}\\n\equiv b\pmod q}}\frac{f(n)\varphi(n)}{(n+\alpha)^{\sigma}}\bigg|\leq \max(0,|\Lambda(b)|-S_{3,b}).
		\end{equation}  
		We introduce the notation 
		\[
		S_{1,b}=\bigg|\sum_{\substack{n=0\\n\equiv b\pmod q}}^{N_j}\frac{f(n)\varphi(n)}{(n+\alpha)^{\sigma}}\bigg|,\quad S_{4,b}=|f(b)|\sum_{\substack{n>N_{j+1}\\n\equiv b\pmod q}}\frac{1}{(n+\alpha)^{\sigma}},\quad S_{2,b}=|f(b)|\sum_{n\in\g{B}(b)}\frac{1}{(n+\alpha)^{\sigma}}. \]
		Now, recalling that $\g{B}(b)$ contains at most $0.46\frac{M_j}{q}$ elements and $\g{A}(b)$ at least $0.54 \frac{M_j}{q}$, we have
		\[
		\frac{S_{3,b}}{S_{2,b}}\geq \frac{54}{46}\frac{(N_j+\alpha)^{\sigma}}{(N_{j+1}+\alpha)^{\sigma}}>\frac{101}{99},
		\]
		Thus, we deduce
		\begin{equation}\label{3-2}
		S_{3,b}-S_{2,b}	>\frac{1}{100}(S_{3,b}+S_{2,b}).	
		\end{equation}
		Now, by the equations \eqref{induction}, \eqref{N_j+1} and \eqref{3-2} we get 
		\[
		\begin{split}
		\bigg|\sum_{\substack{n=0\\n\equiv b\pmod q}}^{N_{j+1}}\frac{f(n)\varphi(n)}{(n+\alpha)^{\sigma}}\bigg|&
		<\frac{1}{100}S_{4,b}=\frac{1}{100}\sum_{\substack{n>N_{j+1}\\n\equiv b\pmod q}}\frac{|f(n)|}{(n+\alpha)^{\sigma}}.
		\end{split}
		\]
		Summing over the classes modulo $q$, we finally get that 
		\[
		\bigg|\sum_{n=0}^{N_{j+1}}\frac{f(n)\varphi(n)}{(n+\alpha)^{\sigma}}\bigg|<\frac{1}{100}\sum_{b=0}^{q-1}S_{4,b}
		<\frac{1}{100}\sum_{n>N_{j+1}}\frac{|f(n)|}{(n+\alpha)^{\sigma}}.
		\]
		So, equation \eqref{induction} also holds for $j+1$ in place of $j$, as desired. By induction, it then holds for all $j\ge 1$. Since $F(s,f,\alpha)$ is absolutely convergent for $\sigma>1$, the right-hand side goes to zeros as $j\to +\infty$. It then follows that $\sum_{n=0}^{\infty}\frac{f(n)\varphi(n)}{(n+\alpha)^\sigma}=0$ and the proof is complete, since by almost periodicity and Rouché's theorem we can conclude the existence of infinitely many zeros for $F(s,f,\alpha)$ with $\sigma>1$.   
		
		\subsection{Proof of Lemma \ref{lemma}}
		Let $\g{P}$ be the set of the prime ideals $\g{p}$ of $O_K$ defined as in Cassels', with the added condition that $(p,q)=1$, where $p:=\norm(\g{p})$. Then, for any integer $n$ we write 
		\begin{equation}\label{fact}
		(n+\alpha)\g{a}=\g{b}\prod_{\g{p}}\g{p}^{u(\g{p})},
		\end{equation}
		where $u(\g{p})$ is an integer and $\g{b}$ contains all the prime factors of $(n+\alpha)\g{a}$ which are not in $\g{P}$. \\
		Consider now an integer $N>10^6q$ and let 
		$M=\lfloor 10^{-6}N\rfloor$. We define $\g{S}=\g{S}(N,q,b)$
		as the set of the integers $N<n\leq N+M$, $n\equiv b\pmod q$ such that, for all the primes $\g{p}\in\g{P}$ in the factorization \eqref{fact} one has $p^{u(\g{p})}<M$. Let $S=S(N,q,b)=|\g{S}|$. We want an upper bound for $S$.\\
		For any prime $\g{p}\in\g{P}$ and any integer $v$, let $\phi(\g{p}^v,n)$ and $\sigma(n)$ be defined as in \cite{cas}. Thus, the same argument gives, as $N\to\infty$,
      	\begin{equation}\label{sum}
		\sum_{n\in\g{S}}\sigma(n)\geq(2+o(1))S\log M.
		\end{equation}		
		Moreover, by the definition of $\g{P}$, if $\g{p}^v\mid (n_1+\alpha)\g{a}$ and $\g{p}^v\mid (n_2+\alpha)\g{a}$ for some integer $v$ then
		\begin{equation}\label{equiv}
		n_1\equiv n_2\pmod {p^v}.
		\end{equation}
		Since we assumed $(p,q)=1$, by the Chinese remainder theorem $n_1\equiv n_2\pmod {p^vq}$.		As in \cite{cas}, we get
		\begin{equation}\label{s_2}
		\sum_{n\in\g{S}}\phi(\g{p}^v,n)\leq \sum_{\substack{N<n\leq N+M\\n\equiv b\pmod q}}\phi(\g{p}^v,n)\leq \bigg(\frac{M}{p^vq}+1\bigg)\log p,
		\end{equation}
		and, assuming $\g{p}_1\ne \g{p}_2$,
		\begin{equation}\label{s_3}
		\sum_{n\in\g{S}}\phi(\g{p}_1^v,n)\phi(\g{p}_2^v,n)\leq \sum_{\substack{N<n\leq N+M\\n\equiv b\pmod q}}\phi(\g{p}_1^v,n)\phi(\g{p}_2^v,n)\leq \log p_1\log p_2\bigg(\frac{M}{p_1p_2q}+1\bigg).
		\end{equation}
		Writing $\sigma(n)=\sigma_1(n)+\sigma_2(n)+\sigma_3(n)$, with the same notation of \cite{cas}, using the prime ideal theorem, partial summation and equations \eqref{s_2}, \eqref{s_3}, we get
		\[\sum_{n\in\g{S}}\sigma_2(n)
		\leq \bigg(\frac{1}{2}+o(1)\bigg)\frac{M}{q}\log M,\]
		\[\sum_{n\in\g{S}}(\sigma_3(n))^2\leq\bigg(\frac{3}{8}+o(1)\bigg)\frac{M}{q}\log^2 M,\]
		and 
		\[\sum_{n\in\g{S}}\sigma_1(n)=O(M)=o(M\log M).\]
		We define $\rho:=\frac{qS}{M}$ and the proof now proceeds exactly as in Cassels'. The better numerical result simply follows by a more precise choice of $\rho$ in expression (37) of \cite{cas}. 
		\section*{Acknowledgments} 
		I would like to thank my PhD supervisors, Alberto Perelli and Sandro Bettin, for their valuable suggestions.


\begin{thebibliography}{99}
			\bibitem{a} T. M. Apostol, \emph{Modular functions and Dirichlet series in Number Theory}, 2nd
			ed., Springer-Verlag, 1990.
			\bibitem{b} H. Bohr, \emph{Om Addition af uendelig mange konvekse Kurver}, Dansk Videnskab. Selsk.
			Forh. 4 (1913), 325–366.
			\bibitem{cas} J.W.S. Cassels, \emph{Footnote to a note of Davenport and Heilbronn}, J. London Math. Soc. 36 (1961), 177-184. 
			\bibitem{c-g}T. Chatterjee, S. Gun,
			\emph{On the zeros of generalized Hurwitz zeta functions},
			J. Number Theory, 145 (2014), 352-361.
			\bibitem{d-h} H. Davenport, H. Heilbronn, \emph{On the zeros of certain Dirichlet series I, II}, Journal London Math. Soc. 11 (1936), 181-185, 3017-312.
			\bibitem{s-w} E. Saias, A. Weingartner, \emph{Zeros of Dirichlet series with periodic coefficients}, Acta Arithmetica 140 (2009), 335-344. 
		\end{thebibliography}
	\end{document}